\documentclass[11pt,sumlimits,intlimits]{amsart}
\usepackage{enumerate}
\usepackage{amscd}
\usepackage{amssymb}
\usepackage[matrix,arrow]{xy}
\usepackage{centernot}
\usepackage{mathtools}
\usepackage{ stmaryrd }
\usepackage[pagebackref,colorlinks,linkcolor=red,citecolor=blue,urlcolor=blue,hypertexnames=true]{hyperref}
\usepackage[pagebackref,hypertexnames=true]{hyperref}

\newtheorem{theorem}{Theorem}[section]

\newtheorem{proposition}[theorem]{Proposition}

\newtheorem{lemma}[theorem]{Lemma}

\newtheorem{corollary}[theorem]{Corollary}

\newtheorem{definition}[theorem]{Definition}

\newcommand{\D}{\mathcal{D}}

\newcommand{\C}{\mathcal{C}}
\newcommand{\DT}{\mathcal{D}_T}
\newcommand{\CT}{\mathcal{C}_T}
\newcommand{\DL}{\mathcal{D}_L}
\newcommand{\CL}{\mathcal{C}_L}
\newcommand{\id}{\mathrm{id}}

\newcommand{\Manoa}{M\=anoa}

\newcommand{\Hawaii}{Hawai\kern.05em`\kern.05em\relax i}

\begin{document}
\message{TIME: \pdfelapsedtime^^J}
\title[Localization $C^{*}$-algebras and $K$-theoretic duality]{Localization $C^{*}$-algebras and $K$-theoretic duality }
\author{Marius Dadarlat}\address{MD: Department of Mathematics, Purdue University, West Lafayette, IN 47907, USA}\email{mdd@purdue.edu}
\author{Rufus Willett}\address{RW:Department of Mathematics, University of \Hawaii~at \Manoa,  2565 McCarthy Mall, Keller 401A
Honolulu
HI 96822
USA}\email{rufus@math.hawaii.edu}
\author{Jianchao Wu}\address{JW: Department of Mathematics, Penn State University, 109 McAllister Building, University Park, PA 16802}\email{jianchao.wu@psu.edu}
\thanks{M.D. was partially supported by NSF grants \#DMS--1362824 and \#DMS--1700086.  R.W. was partially supported by NSF grants \#DMS--1401126 and \#DMS--1564281. J.W. was partially supported by SFB 878 {\em Groups, Geometry and Actions} and ERC grant no.\ 267079.}	

\begin{abstract}
 Based on the localization algebras of Yu, and their subsequent analysis by Qiao and Roe, we give a new picture of $KK$-theory in terms of time-parametrized families of (locally) compact operators that asymptotically commute with appropriate representations.
 \end{abstract}


\maketitle
\section{introduction}
Let $A$ be a unital $C^*$-algebra, unitally represented on a Hilbert space $H$.  Assume that there is a continuous family $(q_t)_{t\in [0,\infty)}$ of compact projections on $H$ that asymptotically commutes with $A$, meaning that $[q_t,a]\to 0$ as $t\to\infty$ for all $a\in A$.  Note that if $p$ is a projection in $A$, then the family $t\mapsto p q_t$ of compact operators gets close to being a projection, and is thus close to a projection that is uniquely defined up to homotopy; in particular, there is a well-defined $K$-theory class $[p q_t]\in K_0(K(H))=\mathbb{Z}$.  It is moreover not difficult to see that this idea can be bootstrapped up to define a homomorphism
\begin{equation}\label{eq:pair}
[q_t]:K_0(A)\to \mathbb{Z}, \quad [p]\mapsto [pq_t].
\end{equation}
This suggests using such parametrized families $(q_t)_{t\in [0,\infty)}$ to define elements of $K$-homology.

Indeed, something like this has been done when $A=C(X)$ is commutative. In this case, the condition that $[q_t,a]\to 0$ is equivalent to the condition that the `propagation' of $q_t$ (in the sense of Roe, \cite[Ch. 6]{HigRoe:khomology}) tends to zero, up to an arbitrarily good approximation.  Motivated by considerations like the above, and by the heat kernel approach to the Atiyah-Singer index theorem, Yu \cite{Yu:localization-algebra} described $K$-homology for simplicial complexes in terms of families with asymptotically vanishing propagation using his localization algebras.  Subsequently, Qiao and Roe \cite{Qiao-Roe} gave a new approach to this result of Yu that works for all compact (in fact, all proper) metric spaces.

In this paper, we present a new picture of Kasparov's $KK$ groups based on asymptotically commuting families.  Thanks to the relationship between asymptotically vanishing propagation and asymptotic commutation, our picture can be thought of as an extension of the results of Yu and Qiao-Roe from commutative to general (separable) $C^*$-algebras, and from $K$-homology to $KK$-theory.   We think this gives an attractive picture of $KK$-theory.  We also suspect that the ease with which the pairing in line \eqref{eq:pair} is defined \textemdash~ note that unlike in the case of Paschke duality, there is no dimension shift, and unlike in the case of $E$-theory, there is no suspension \textemdash~ should be useful for future applications.  Having said this, we should note that the picture of the pairing in line \eqref{eq:pair} is overly simplified, as in general to get the whole $KK$ group one needs to consider formal differences of such families of projections $(q_t)$ in an appropriate sense.

We now give precise statements of our main results.   For a $C^{*}$-algebra $B$, we denote by $C_u(T,B)$ the $C^{*}$-algebra of bounded and uniformly continuous functions from $T=[0,\infty)$ to $B$.
Inspired by work of Yu ~\cite{Yu:localization-algebra} and Qiao and Roe \cite{Qiao-Roe}, we define the localization algebra
 $\CL(\pi)$ associated to a representation $\pi$ of a separable $C^{*}$-algebra $A$ on a separable Hilbert space to be the $C^{*}$-subalgebra of $C_u(T,L(H))$ consisting
of all the functions $f$ such that for all $a\in A$,
\[ [f,\pi(a)]\in C_0(T,K(H))\,\, \text{and} \, \,\pi(a)f\in C_u(T,K(H)).\]
Let us recall that a representation $\pi$ is ample if it is nondegenerate, faithful and $\pi(A)\cap K(H)=\{0\}$.
One verifies that the isomorphism class of $\CL(\pi)$ does not depend on the choice of an ample representation $\pi$. In this case, we write $\C_L(A)$ in place of $\CL(\pi)$ and view $A$ as a $C^{*}$-subalgebra of $L(H)$. Note that if $A$ is unital, then
\[\C_L(A)=\{f\in C_u(T,K(H))\colon [f,a]\in C_0(T,K(H)),\,\forall a\in A\}.\]

In this paper we establish canonical isomorphisms $K^i(A) \cong K_i(\C_L(A)) $, $i=0,1$, between the $K$-homology of $A$ and the $K$-theory of
the localization algebra $\C_L(A)$.
More generally, we use results of Thomsen \cite{Thomsen} to show that for separable $C^{*}$-algebras $A$, $B$
and any absorbing representation $\pi:A \to L(H_B)$ on the standard infinite dimensional countably generated right Hilbert $B$-module $H_B$, there are canonical isomorphisms of groups
\begin{equation}\label{eq:kk iso}
\xymatrix{ KK_i(A,B) \ar[r]^-\cong & K_i(\CL(\pi)),\quad i=0,1,}
\end{equation}
where  the localization $C^{*}$-algebra $\CL(\pi)$ consists of those functions $f\in C_u(T,L(H_B))$ such that for all $a\in A$,
\[ [f,\pi(a)]\in C_0(T,K(H_B))\,\, \text{and} \, \,\pi(a)f\in C_u(T,K(H_B)).\]
The isomorphism in line \eqref{eq:kk iso} is defined and proved by combining Paschke duality with a generalization of the techniques used by Roe and Qiao in the commutative case. \\

The paper is structured as follows.  In Section~\ref{sec:abs}, we discuss absorbing representations and give a version of Voiculescu's theorem appropriate to localization algebras.  In Section~\ref{sec:1}, we define the various dual algebras and localization algebras that we use, and show that they do not depend on the choice of absorbing representation.  In Section~\ref{sec:dual}, we prove the isomorphism in line \eqref{eq:kk iso}.  Finally, in Section~\ref{sec:inv}, we construct maps $K_i(\CL(\pi))\to E_i(A,B)$ and show that they `invert' the isomorphism in line \eqref{eq:kk iso} in the sense that the composition $KK_i(A,B)\to K_i(\CL(\pi)) \to E_i(A,B)$ is the canonical natural transformation from $KK$-theory to $E$-theory.

\paragraph{\textbf{Acknowledgements:}} Part of this research was conducted during the authors' visits to the University of M\"{u}nster, the University of \Hawaii~at \Manoa, and the Institut Mittag-Leffler. We are grateful for the hospitality of the host institutes.  We would also like to thank the referee for a close reading of the paper, and several useful suggestions.

\section{Absorbing representations}\label{sec:abs}
Let $A$, $B$ be separable $C^{*}$-algebras.
If $E, F$ are countably generated right Hilbert $B$-modules, we denote by $L(E,F)$ the $C^{*}$-algebra of  bounded $B$-linear adjointable  operators from $E$ to $F$. The corresponding $C^{*}$-algebra of ``compact" operators is denoted by $K(E,F)$, \cite{Kas:cp}.
Set $L(E)=L(E,E)$ and $K(E)=K(E,E)$.
Recall that $H_B$ is the standard infinite dimensional countably generated right Hilbert $B$-module.

We shall use the notion of (unitally) absorbing $*$-representations $\pi:A \to L(H_B)$, see \cite{Thomsen}.

\begin{definition}\label{def:absorbing}
(i) Suppose that $A$ is a unital separable $C^{*}$-algebra. A unital representation $\pi:A \to L(H_B)$ is called \emph{unitally absorbing} for the pair $(A,B)$ if  for any other unital representation
$\sigma:A \to L(E)$, there is an isometry $v\in C_b(\mathbb{N}, L(E,H_B))$ such that
$v\sigma(a)-\pi(a)v\in C_0(\mathbb{N}, K (E,H_B))$ for all $a\in A$.

(ii) Suppose that $A$ is a separable $C^{*}$-algebra. We denote by $\widetilde{A}$ the unitalization of $A$, with the convention that
$\widetilde{A}=A$, if $A$ is already unital.
A  representation $\pi:A \to L(H_B)$ is called
\emph{absorbing} for the pair $(A,B)$ if its unitalization $\widetilde{\pi}:\widetilde{A} \to L(H_B)$
is unitally absorbing  for the pair $(\widetilde{A},B)$.
\end{definition}
Note that in Definition~\ref{def:absorbing},
if we denote the components of $v$ by $v_n$, we have
$v_n\sigma(a)-\pi(a)v_n\in K (E,H_B)$ and $\lim_{n\to \infty} \|v_n\sigma(a)-\pi(a)v_n\|=0$ for all $a\in A$.
\begin{theorem}[Voiculescu, \cite{Voi:Weyl-vn}]\label{thm:Voiculescu-classic}
Any ample representation of a separable $C^{*}$-algebra on a separable infinite dimensional Hilbert space is absorbing.
\end{theorem}
\begin{theorem}[Kasparov, \cite{Kas:cp}]\label{thm:Kasparov-abs}
Let $A$ be a unital separable $C^{*}$-algebra and let $B$ be a $\sigma$-unital $C^{*}$-algebra.
If either $A$ or $B$ are nuclear, then any unital ample representation $\pi:A \to L(H)\subset L(H_B)$ is  absorbing for the pair $(A,B)$.
\end{theorem}
\begin{theorem}[Thomsen, \cite{Thomsen}]\label{thm:Thomsen-abs-exist}
For any separable $C^{*}$-algebras $A$ and $B$ there exist absorbing representations $\pi:A \to L(H_B)$.
\end{theorem}

Given two $*$-representations $\pi_i:A \to L(E_i)$ we write that $\pi_1\underset{v}\preccurlyeq \pi_2$ if there is an isometry $v\in C_u(T,L(E_1,E_2))$ such that
\begin{equation*}\label{eqn:sVoic}
v\pi_1(a)-\pi_2(a)v\in C_0(T,K(E_1,E_2)).
\end{equation*}
If in addition $v\in C_u(T,L(E_1,E_2))$ is a unitary with the same property, then we write $\pi_1\underset{v}\approx \pi_2$.

Let $w^\infty:E_1^\infty \to E_1 \oplus E_1^\infty$ be the unitary defined by $w^\infty (h_0,h_1,h_2,...)=h_0 \oplus (h_1,h_2,...)$.
\begin{lemma}[Lemma 2.16, \cite{DadEil:class}]\label{lemma:DE-KK}
 Let $\pi_i:A \to L(E_i)$ be two representations and let $v\in L(E^{\infty}_1,E_2)$
be an isometry such that $v\pi_1^\infty(a)-\pi_2(a)v\in K(E^\infty_1,E_2)$ for all $a\in A$.
Then  $u=(1_{E_1}\oplus v)w^\infty v^*+(1_{E_2}-v v^*)\in L(E_2, E_1\oplus E_2)$
is a unitary operator such that $\pi_1(a)\oplus \pi_2(a)-u\pi_2(a)u^*\in K(E_1\oplus E_2)$
for all $a\in A$ and moreover
\[\|\pi_1(a)\oplus \pi_2(a)-u\pi_2(a)u^*\|\leq 6\|v\pi_1^\infty(a)-\pi_2(a)v\| +4\|v\pi_1^\infty(a^*)-\pi_2(a^*)v\|.\]
\end{lemma}
Using this lemma, one obtains the following strengthened variation of Voiculescu's theorem \cite{Voi:Weyl-vn}. This result appears in \cite{DadEil:AKK}
as Theorem 3.11, except that the uniform continuity of the isometry $v$ and the unitary $u$ were not addressed explicitly in the statement.
\begin{theorem}\label{thm:Voiculescu}
Let $A$, $B$ be separable $C^{*}$-algebras and
let $\pi_i:A \to L(E_i)$, $i=1,2$ be two representations where $E_i\cong H_B$.
If $\pi_2$ is absorbing, then  $\pi_1\underset{v}\preccurlyeq \pi_2$ for some isometry $v\in C_u(T,L(E_1,E_2))$.
If both $\pi_1$ and $\pi_2$ are absorbing, then $\pi_1\underset{u}\approx \pi_2$ for some unitary $u\in C_u(T,L(E_1,E_2))$.
\end{theorem}

\begin{proof} Since $\pi_2$ absorbs $\pi^\infty_2$ there is an isometry $u=(u_n)_n\in C_b(\mathbb{N},L(E_2^\infty,E_2))$ such that
$u\pi_2^\infty(a)-\pi_2(a)u\in C_0(\mathbb{N},K(E_2^\infty,E_2))$ for all $a\in A$.
Since $\pi_2$  absorbs $\pi_1$, there is a sequence of isometries $w_n \in L(E_1,E_2^\infty)$ with mutually orthogonal ranges
such that $w_n\pi_1(a)-\pi_2^\infty(a)w_n\in K (E_1,E_2^\infty)$ and $\lim_{n\to \infty} \|w_n\pi_1(a)-\pi_2^\infty(a)w_n\|=0$ for all $a\in A$. Then $v_n =u_n w_n \in L(E_1,E_2)$ is a sequence of isometries
with orthogonal ranges such that the corresponding isometry $v\in C_b(\mathbb{N}, L(E_1,E_2))$ satisfies
$v\pi_1(a)-\pi_2(a)v\in C_0(\mathbb{N},K(E_1,E_2))$ for all $a\in A$. This follows from the identity
\[u_n w_n\pi_1(a)-\pi_2(a)u_n w_n=u_n(w_n\pi_1(a)-\pi_2^\infty(a)w_n)+(u_n\pi_2^\infty(a)-\pi_2(a)u_n)w_n.\]
Since $v_n^*v_m=0$ for $n\neq m$,
 one observes that
 $\mathbf{v}(n+s)=(1-s)^{1/2}v_n+s^{1/2}v_{n+1}$, $0\leq s \leq 1$, extends $v$ to a uniformly continuous  isometry $\mathbf{v}\in C_u(T,L(E_1,E_2))$ that satisfies $\pi_1\underset{\mathbf{v}}\preccurlyeq \pi_2$.

For the second part of the statement, we note that by the
  first part $\pi^\infty_1\underset{v}\preccurlyeq \pi_2$. Thus,
  $ v\pi_1^\infty(a)-\pi_2(a)v\in C_0(T,K(E^\infty_1,E_2))$, for all $a\in A$ where $v=(v_t)_{t\in T}$ is a uniformly continuous isometry with $v_t\in L(E^{\infty}_1,E_2)$.
It follows by  Lemma~\ref{lemma:DE-KK}  that
\[u_t=(1_{E_1}\oplus v_t)w^\infty v_t^*+(1_{E_2}-v_t v_t^*)\]
is a uniformly continuous unitary such that $\pi_1\oplus \pi_2 \underset{u}\approx \pi_2$. By symmetry we have that $\pi_1\oplus \pi_2 \underset{u}\approx \pi_1$ and hence $\pi_1 \underset{u}\approx \pi_2$.
\end{proof}

\section{Dual algebras}\label{sec:1}
Let $A$, $B$ be separable $C^{*}$-algebras and let
 $\pi:A \to L(H_B)$ be a $*$-representation.
 \begin{definition}
   The localization algebra $\CL(\pi)$ associated to $\pi$ is the $C^{*}$-subalgebra of $C_u(T,L(H_B))$ consisting
of all the functions $f$ such that\\
$[f,\pi(a)]\in C_0(T,K(H_B))$ and  $\pi(a)f\in C_u(T,K(H_B))$ for all $a\in A$.
 \end{definition}
 While $\CL(\pi)$ is the central object of the paper, we also need to consider
a series of  pairs of $C^{*}$-algebras and ideals which will play a supporting role:
\[\D(\pi)=\{b\in L(H_B)\colon [b,\pi(a)]\in K(H_B),\,\forall a\in A\},\]
\[\C(\pi)=\{b\in L(H_B)\colon  \pi(a)b\in K(H_B),\,\forall a\in A\},\]
and their parametrized versions,
\[\DT(\pi)=\{f\in C_u(T,L(H_B))\colon [f,\pi(a)]\in C_u(T,K(H_B)),\,\forall a\in A\}\cong C_u(T,\D(\pi)),\]
\[\CT(\pi)=\{f\in C_u(T,L(H_B))\colon  \pi(a)f\in C_u(T,K(H_B)),\,\forall a\in A\}\cong C_u(T,\C(\pi)).\]
The evaluation map at $0$ leads to the pair
\[\DT^0(\pi)=\{f\in \DT(\pi)\colon f(0)={0}\},\]
\[\CT^0(\pi)=\{f\in \CT(\pi)\colon f(0)={0}\}.\]

Finally, we view the localization algebra $\CL(\pi)$ as an ideal of
\[\DL(\pi)=\{f\in C_u(T,L(H_B))\colon [f,\pi(a)]\in C_0(T,K(H_B)),\,\forall a\in A\},\]
\[\CL(\pi)=\{f\in \DL(\pi)\colon \pi(a)f\in C_u(T,K(H_B)), \, \forall a\in A\}.\]

To simplify some of the statements it is useful to introduce the following notation:
$A_1(\pi)=\DT(\pi)$, $A_2(\pi)=\CT(\pi)$, $A_3(\pi)=\DT^0(\pi)$, $A_4(\pi)=\CT^0(\pi)$, $A_5(\pi)=\DL(\pi)$ and $A_6(\pi)=\CL(\pi)$.
We are going to see that the isomorphism classes of these $C^{*}$-algebras are independent of $\pi$, provided that $\pi$ is an absorbing representation.
We follow the presentation from \cite[Section 5.2]{HigRoe:khomology} where  analogous properties of $\D(\pi)$ and $\C(\pi)$ are established, except that we need to employ a strengthened version of Voiculescu's theorem, contained in Theorem~\ref{thm:Voiculescu} above.

Let $\pi_1,\pi_2:A \to L(H_B)$ be two representations.
\begin{lemma} If $\pi_1\underset{v}\preccurlyeq \pi_2$, then the equation $\Phi_v(f)=v f v^*$ defines
 a $*$-homomorphism
$\Phi_v:\DT(\pi_1)\to \DT(\pi_2)$ with the property that
$\Phi_v(A_j(\pi_1))\subset A_j(\pi_2)$ for all $1 \leq j \leq 6$.
\end{lemma}
\begin{proof}  This follows from the identities:
\[  [vfv^*,\pi_2(a)]=v[f,\pi_1(a)]v^*+(v\pi_1(a)-\pi_2(a)v)f v^*-v f(v\pi_1(a^*)-\pi_2(a^*)v)^*,\]
\[\pi_2(a)v f v^*=v\pi_1(a)f v^*-(v\pi_1(a)-\pi_2(a)v)f v^*.\qedhere\]
\end{proof}
\begin{corollary}\label{cor:indep}
 Let $\pi_1,\pi_2:A \to L(H_B)$ be  two absorbing representations. Then $A_j(\pi_1)\cong A_j(\pi_2)$ for all $1 \leq j \leq 6$.
\end{corollary}
\begin{proof}  Proposition~\ref{thm:Voiculescu} yields a unitary
 $v\in C_u(T,L(H_B))$ such that $\pi_1\underset{v}\approx \pi_2$.
 The corresponding maps $\Phi_v: A_j(\pi_1)\to A_j(\pi_2)$ are isomorphisms.
\end{proof}

\begin{lemma}\label{lemma:K-theory-natural}
Let $\pi_1,\pi_2:A \to L(H_B)$ be two representations of $A$ and
suppose that $v_1,v_2$ are two isometries such that
$\pi_1\underset{v_i}\preccurlyeq \pi_2$, $i=1,2$. Then $(\Phi_{v_1})_*=(\Phi_{v_2})_*:K_*(A_j(\pi_1))\to K_*(A_j(\pi_2))$
 for all $1\leq j \leq 6$.
\end{lemma}
\begin{proof}
The unitary $u=\begin{pmatrix} 1-v_1 v_1^* &v_1 v_2^*\\v_2 v_1^* &  1-v_2 v_2^* \end{pmatrix}\in M_2(\DL(\pi_2))$ conjugates  $\begin{pmatrix}\Phi_{v_1}&0\\0 &  0 \end{pmatrix}$ over  $\begin{pmatrix}0&0\\0 &  \Phi_{v_1}\end{pmatrix}.$
It follows that $(\Phi_{v_1})_*=(\Phi_{v_2})_*:K_*(\DT(\pi_1))\to K_*(\DT(\pi_2))$.
Similarly, one verifies that the equality $(\Phi_{v_1})_*=(\Phi_{v_2})_*:K_*(A_j(\pi_1))\to K_*(A_j(\pi_2))$
holds for all $1\leq j \leq 6$.
\end{proof}

Denote by $\pi^\infty$ the direct sum $\pi^\infty =\bigoplus_{n=1}^\infty \pi : A \to L(H_B^\infty)=L(\bigoplus_{n=1}^\infty H_B)$.
\begin{corollary}\label{cor:inclusion} If $\pi:A \to L(H_B)$ is an absorbing representation, then
the inclusion $\DT(\pi)\to \DT(\pi^\infty)$, $f \mapsto(f,0,0,...)$ induces  isomorphisms on $K$-theory:
$K_*(A_j(\pi))\to K_*(A_j(\pi^\infty))$, for all $1 \leq j \leq 6$.
\end{corollary}
\begin{proof}
 We have $\pi\underset{v}\preccurlyeq \pi^\infty$, where $v\in C_u(T,L(H_B,H_B^\infty))$ is the constant isometry defined by  $v(t)(h)=(h,0,0,...)$ for any $t \in T$ and $h \in H_B$. The inclusion map from the statement coincides with $\Phi_v$.
 On the other hand $\pi\underset{u}\approx \pi^\infty$ since $\pi$ is absorbing and hence
  $\Phi_u$ is an isomorphism. We conclude the proof by noting that
  $(\Phi_v)_*=(\Phi_u)_*$ by Lemma~\ref{lemma:K-theory-natural}.
\end{proof}

\section{A duality isomorphism}\label{sec:dual}
Let $A$ and $B$ be separable $C^{*}$-algebras. We are going to show that when we fix an absorbing representation $\pi \colon A \to L(H_B)$ (the existence of such an absorbing representation is guaranteed by Theorem~\ref{thm:Thomsen-abs-exist}), the $K$-theory of $\CL(\pi)$ is canonically isomorphic to the $KK$-theory of the pair $(A,B)$.

We start with a technical lemma that will be used several times later.

\begin{lemma}\label{lem:qc-unit}
For any separable $C^*$-algebra $D\subset C_u(T,L(H_B))$ there is a positive contraction $x\in C_u(T,K(H_B))$ such that:
\begin{enumerate}[(a)]
\item $[x,d]\in C_0(T,K(H_B))$ for all $d\in D$, and
\item $(1-x)d\in C_0(T,K(H_B))$ for all $d\in D\cap C_u(T,K(H_B))$.
\end{enumerate}
\end{lemma}

\begin{proof} Our arguments will in fact show that the statement holds true in the more general situation where $L(H_B)$ is replaced by a $C^{*}$-algebra $L$ and $K(H_B)$ is replaced by a two-sided closed ideal $I$ of $L$.
Let $\dot{D}$ denote the $C^*$-subalgebra of $L$ generated by all images $d(t)$ as $d$ ranges over $D$ and $t$ over $T$.  This is separable, and contains $\dot{C}=\dot{D}\cap I$ as an ideal.  Let $(x_n)_n$ be a positive contractive approximate unit for $\dot{C}$ which is quasi-central in $\dot{D}$.  Choose countable dense subsets $(d_k)_{k=1}^\infty$ and $(c_k)_{k=1}^\infty$ of $D$ and $D\cap C_u(T,I)$ respectively.  As for each $n$, the subsets $\bigcup_{k=1}^n\{d_k(t):t\in [0,n+1]\}$ and $\bigcup_{k=1}^n\{c_k(t):t\in [0,n+1]\}$ of $\dot{D}$ and $\dot{C}$ respectively are compact, we may assume on passing to a subsequence of $(x_n)$ that
\begin{enumerate}[(i)]
\item $\|[d_k(t),x_n]\|<\frac{1}{n+1}$ for all $1\leq k\leq n$ and all $t\in [0,n+1]$, and
\item $\|(1-x_n)c_k(t)\|<\frac{1}{n+1}$ for all $1\leq k \leq n$ and all $t\in [0,n+1]$.
\end{enumerate}
For $t\in [n,n+1)$, write $s=t-n$ and set $x(t)=(1-s)x_n+s x_{n+1}$; note that the function $x:t\mapsto x(t)$ is uniformly continuous.  Then from (i) and (ii) above we have
\begin{enumerate}[(i)]
\item $\|[d_k(t),x(t)]\|<\frac{1}{n+1}$ for all $1\leq k\leq n$ and all $t\in [n,n+1)$, and
\item $\|(1-x(t))c_k(t)\|<\frac{1}{n+1}$ for all $1\leq k \leq n$ and all $t\in [n,n+1)$.
\end{enumerate}
This implies that $x$ has the right properties.
\end{proof}

We have {obvious} inclusions $\DL(\pi)\subset \DT(\pi)$ and $\CL(\pi)\subset \CT(\pi)$ which induce
a $*$-homomorphism
\[\eta:  \DL(\pi)/\CL(\pi) \to\DT(\pi)/\CT(\pi).\]

\begin{proposition}\label{prop:1} For any separable $C^{*}$-algebras $A, B$ and any representation $\pi:A \to L(H_B)$, the map $\eta$ is a $*$-isomorphism.
\end{proposition}

\begin{proof}
It is clear from the definitions that $\CL(\pi)=\DL(\pi)\cap \CT(\pi)$ and hence  $\eta$ is {injective}.
It remains to prove that $\eta$ is surjective. It  suffices to show that for any
$f\in \DT(\pi)$  there is  $\tilde{f}\in  \DL(\pi)$ such that $\tilde{f}-f\in \CT(\pi)$.
Let $f\in \DT(\pi)$  be given.

Let $D$ be the $C^*$-subalgebra of $C_u(T,L(H_B))$ generated by $\pi(A)$ (embedded as constant functions) and $f$, and let $x$ be as in Lemma~\ref{lem:qc-unit}.  With this choice of $x$ (that depends on $f$) we define
 $\tilde{f}=(1-x)f.$
 Note that $\tilde{f}=f - x f\in \DT(\pi)$ since $f,x\in \DT(\pi)$, and $\tilde{f}-f=-x f\in C_u(T,K(H_B))$ since
  $x\in C_u(T, K(H_B))$.
In particular it follows that $\tilde{f}-f \in \CT(\pi)$.

It remains to verify that $\tilde{f}\in \DL(\pi)$.  This follows as for any $a\in A$,
\[[\tilde{f},\pi(a)]=[(1-x)f,\pi(a)]=[\pi(a),x]f+(1-x)[f,\pi(a)]. \qedhere\]

\end{proof}

An adaptation of the arguments from the paper  \cite{Qiao-Roe}
of   Qiao and Roe gives:
\begin{proposition}\label{prop:2} Let $A,B$ be separable $C^{*}$-algebras and let $\pi:A \to L(H_B)$ be an absorbing representation. Then
\begin{itemize}
\item[(a)] $K_*(\DL(\pi))=0$ and hence the boundary map
\newline
\noindent$\partial:K_*(\DL(\pi)/\CL(\pi))\to K_{*+1}(\CL(\pi))$
is an isomorphism. 
\item[(b)] The evaluation map at $t=0$ induces an isomorphism. 
\newline
\noindent $e_*:K_*(\DT(\pi)/\CT(\pi))\to K_*(\D(\pi)/\C(\pi))$.
\end{itemize}
\end{proposition}
\begin{proof}
Fix an ample representation $\pi$ of $A$. One  verifies that if $f\in \DL(\pi)$, then the formula
\[F(t):=(f(t),f(t+1),...,f(t+n),...)\]
defines an element $F\in \DL(\pi^\infty)$.
Indeed,
\[[F(t),\pi(a)]=([f(t),\pi(a)],[f(t+1),\pi(a)],...,[f(t+n),\pi(a)],...)\]
and  each entry belongs to $C_0(T,K(H_B))$ and is bounded by $\|[f,\pi(a)]\|$.
This shows that $[F,\pi(a)]\in C_u(T,K(H_B^{\infty}))$. Since $[f,\pi(a)]\in C_0(T,K(H_B))$,
it follows immediately that in fact $[F,\pi(a)]\in C_0(T,K(H_B^{\infty}))$.

With these remarks, the proof of (a) goes just like the proof of Proposition 3.5 from \cite{Qiao-Roe}.
Indeed, define  $*$-homomorphisms  $\alpha_i:\DL(\pi)\to \DL(\pi^\infty)$, $i=1,2,3,4$ by
\[\alpha_1(f)=(f(t),0,0,...), \]
\[\alpha_2(f)=(0,f(t+1),f(t+2),...), \]
\[\alpha_3(f)=(0,f(t),f(t+1),...) \]
\[\alpha_4(f)=(f(t),f(t+1),f(t+2),...) .\]
It is clear that $\alpha_1+\alpha_2=\alpha_4$.
The isometry $v\in L(H_B^\infty)$ defined by  $v(h_0,h_1,h_2,...)=(0,h_0,h_1,h_2,...)$ commutes with $\pi^\infty(A)$ and hence $v\in \DL(\pi^\infty)$.
 Moreover
$\alpha_4(a)=v\alpha_3(a)v^*$  and hence $(\alpha_4)_*=(\alpha_3)_*$ by \cite[Lemma 4.6.2]{HigRoe:khomology}.  Using uniform continuity, one shows that $\alpha_3$ is homotopic to $\alpha_2$, via the homotopy $f(t)\mapsto (0,f(t+s),f(t+s+1),...)$, $0\leq s \leq 1$.
We deduce that
\[(\alpha_1)_*+(\alpha_2)_*=(\alpha_1+\alpha_2)_*=(\alpha_4)_*=(\alpha_3)_*=(\alpha_2)_*\]
and hence $(\alpha_1)_*=0$. This concludes the proof of (a), since $(\alpha_1)_*$ is an isomorphism  by Corollary~\ref{cor:inclusion}.

(b) One follows the proof of Proposition 3.6 from  \cite{Qiao-Roe} to show that both $K_*(\DT^0(\pi))=0$ and $K_*(\CT^0(\pi))=0$. The desired conclusion will then follow in view of the
 split exact sequence:
\[
\xymatrix{
0 \ar[r] &\DT^0(\pi)/\CT^0(\pi) \ar[r] & \DT(\pi) /\CT(\pi)\ar[r] &\D(\pi)/\C(\pi) \ar[r] & 0.
}
\]
Any $f\in \DT^0(\pi)$ can be extended by $0$ to an element of $C_u(\mathbb{R},L(H_B))$.
With this convention,
define four maps $\beta_i:\DT^0(\pi)\to \DT^0(\pi^\infty)$, $i=1,2,3,4$ by
\[\beta_1(f)=(f(t),0,0,...), \]
\[\beta_2(f)=(0,f(t-1),f(t-2),...), \]
\[\beta_3(f)=(0,f(t),f(t-1),...) \]
\[\beta_4(f)=(f(t),f(t-1),f(t-2),...) .\]
This definition requires that one verifies that if $f \in \DT^0(\pi)$,
then \[F'(t):=(f(t),f(t-1),...,f(t-n),...)\]
defines an element of $ \DT^0(\pi^\infty)$. This is clearly the case, since if $f$ is uniformly continuous, then so is $F'$  and moreover,
just as argued in \cite{Qiao-Roe},  for each $t$ in a fixed bounded interval only finitely many components of $F'(t)$  are non-zero, and hence $[F'(t),\pi^\infty(a)]\in K(H_B^\infty)$  if $[f(t),\pi(a)]\in K(H_B)$ for all $t\in T$.
Note that $(\beta_4)_*=(\beta_3)_*$ since
$\beta_4(a)=v\beta_3(a)v^*$ where $v\in \DT(\pi^\infty)$ is the same isometry as in part (a).
Using uniform continuity, one observes that $\beta_3$ is homotopic to $\beta_2$, via the homotopy $f(t)\mapsto (0,f(t-s),f(t-s-1),...)$, $0\leq s \leq 1$.
We deduce that
\[(\beta_1)_*+(\beta_2)_*=(\beta_1+\beta_2)_*=(\beta_4)_*=(\beta_3)_*=(\beta_2)_*\]
and hence $(\beta_1)_*=0$. This shows that $K_*(\DT^0(\pi))=0$, since $(\beta_1)_*$ is an isomorphism  by Corollary~\ref{cor:inclusion}. The proof for the vanishing of $K_*(\CT^0(\pi))$ is entirely similar.
Indeed, with the same notation as above, one observes that
 if $f \in \CT^0(\pi)$ then $F'\in \CT^0(\pi^\infty)$. Moreover, the four maps $\beta_i:\DT^0(\pi)\to \DT^0(\pi^\infty)$ restrict to maps  $\beta'_i:\CT^0(\pi)\to \CT^0(\pi^\infty)$ with $\beta'_3$ homotopic to $\beta'_2$ and $(\beta'_1)_*$ is an isomorphism  by Corollary~\ref{cor:inclusion}.
\end{proof}
\begin{theorem}\label{thm:1} Let $A,B$ be separable $C^{*}$-algebras and let $\pi:A \to L(H_B)$ be an absorbing representation. There are canonical isomorphisms of groups
\[\alpha:KK_i(A,B)\stackrel{\cong}\longrightarrow K_i(\CL(\pi)),\quad i=0,1.\]
\end{theorem}
\begin{proof}
Consider the diagram
\[
\xymatrix{
{KK_{i}(A,B)}\ar[r]^-{P} &K_{i+1}(\D(\pi)/\C(\pi))\ar[r]^{\iota_*} & K_{i+1}(\DT(\pi)/\CT(\pi))\ar[d]^{{\eta^{-1}_*}} \\
 & K_{i}(\CL(\pi) )& K_{i+1}(\DL(\pi)/\CL(\pi))\ar[l]_-{\partial} }
\]
where $P$ is the Paschke duality isomorphism,  see \cite{Paschke}, \cite[Remarque~2.8]{Skandalis:K-nuclear}, \cite[Theorem~3.2]{Thomsen}, and $\iota$ is the canonical inclusion.
The maps $\partial$ and ${\iota_*}=e^{-1}_*$ are isomorphisms by Proposition~\ref{prop:2} and $\eta_*$ is an isomorphism by Proposition~\ref{prop:1}.
\end{proof}
As a corollary we obtain the following duality theorem, mentioned in the introduction.  Recall from the introduction that $\CL(A)$ stands for $\CL(\pi)$, where $\pi$ is ample (and thus absorbing, by Theorem~\ref{thm:Voiculescu-classic}), and $A$ is identified with $\pi(A)$.
\begin{theorem}\label{thm:cor} For any separable $C^{*}$-algebra $A$ there are canonical isomorphisms of groups
\(K^i(A)\cong K_i(\CL(A)),\quad i=0,1.\)
\qed
\end{theorem}

\section{An inverse map}\label{sec:inv}
Let $\alpha: KK_i(A,B)\stackrel{\cong}\longrightarrow K_i(\CL(\pi))$ be the isomorphism of Theorem~\ref{thm:1}.
Recall that $K(H_B)\cong B \otimes K(H)$.
Consider the $*$-homomorphism
\[\mathbf{\Phi}: \DL(\pi) \otimes_{\max} A\to \frac{C_u(T,  L(H_B))}{C_0(T,  K(H_B))}\]
defined by $\mathbf{\Phi}(f \otimes a)=f \pi(a) $
and its restriction to $\CL(\pi) \otimes_{\max} A $
\[{\boldsymbol\varphi}: \CL(\pi) \otimes_{\max} A \to \frac{C_u(T,  K(H_B))}{C_0(T,  K(H_B))}.\]

We want $\boldsymbol\varphi$ to define a class in $E$-theory that we can take products with, but have to be a little careful due to the non-separability of the $C^{*}$-algebra $\CL(\pi) \otimes_{\max} A$.   Just as in the case of the $KK$-groups \cite{Skandalis:K-nuclear},
if $C$ is any $C^{*}$-algebra and $B$ is a non-separable $C^{*}$-algebra one defines $E_{sep}(B,C)=\varprojlim_{\,B_1} E(B_1,C)$, with $B_1 \subset B$ and $B_1$ separable. Moreover if $D$ is separable, then $E(D,B)=\varinjlim_{\,B_1} E(D,B_1)$, with $B_1\subset B$ and  $B_1$ separable. With these adjustments, one has a well-defined product
\[E(D,B)\times E_{sep}(B,C)\to E(D,C).\]
Moreover, it is clear that $[[\boldsymbol\varphi]]$ defines an element of the group  $E_{sep}(\CL(\pi) \otimes_{\max} A ,B)$.

Recall the isomorphism $K_i(\CL(\pi))\cong E_i(\mathbb{C}, \CL(\pi))$.
We use the product
\[E_i(\mathbb{C}, \CL(\pi))\times E_{sep}(\CL(\pi) \otimes_{\max} A ,B) \to E_i(A, B)\]
to define a map $\beta: K_i(\CL(\pi)) \to E_i(A, B)$ by $\beta (z) =  [[\boldsymbol\varphi]]\circ (z \otimes \id_A)$.

The map $\beta$ is an inverse of $\alpha$ in the following sense.
\begin{theorem}\label{inverse the} The composition $\beta \circ \alpha$ coincides with the natural map

\noindent $KK_i(A,B)\to E_i(A,B)$, $i=0,1$.
\end{theorem}

\begin{proof} We will give the proof for the odd case $i=1$ and leave the even case for the reader.
Recall that the $E$-theory group $E_1(A,B)$ of Connes and Higson \cite{Con-Hig:etheory} is isomorphic to
$[[SA, K(H_B)]]$ by a desuspension result from \cite{DadLor:unsusp}.

For two continuous functions $f,g:T \to L(H_B)$ we will write $f(s)\sim g(s)$ (or  $f(t)\sim g(t)$) if $f-g\in C_0(T,K(H_B))$.
Let $\{\varphi_s: \CL(\pi) \otimes_{\max} A  \to  K(H_B))\}_{s\in T}$ be an asymptotic homomorphism
representing $\boldsymbol\varphi$. More precisely take $\varphi$ to be a set-theoretic lifting of $\boldsymbol\varphi$.  This means that $\varphi_s( f\otimes a)\sim f(s)\pi(a)$.

The composition $\beta \circ \alpha: KK_1(A,B)\to E_1(A,B)$ is computed as follows.
Let $y\in KK_1(A,B)$ and let $z=Py\in K_0(\D(\pi)/\C(\pi))$ be its image under the Paschke duality isomorphism $P:KK_1(A,B)\to K_0(\D(\pi)/\C(\pi))$.
Let $z$ be represented by a self-adjoint element $e\in \D(\pi)\subset \DT(\pi)$ whose image in
$\D(\pi)/\C(\pi)$ is an idempotent $\dot{e}$. We identify $\D(\pi)$ with the $C^{*}$-subalgebra of constant functions in $\DT(\pi)$. Choose an element $x\in C_u(T,K(H_B))$ as in Lemma~\ref{lem:qc-unit} with respect to the (separable) $C^{*}$-subalgebra $D$ of $C_u(T,L(H_B))$ generated by $\pi(A)$, $e$, and $K(H_B)$.  Therefore both $[x,\pi(a)]$ and $ (1-x)[e,\pi(a)]$ belong to $C_0(T,K(H_B))$ for all $a\in A$, and moreover $(1-x)e\in \DL(\pi)$ as
$$
[(1-x)e,\pi(a)]=[1-x,\pi(a)]e+(1-x)[e,\pi(a)]\in C_0(T,K(H_B))
$$
for all $a\in A$.   Let $e_L=(1-x)e$ and let $\dot{e}_L$ be its image in $ \DL(\pi)/\CL(\pi)$. Under the isomorphism $ \DL(\pi)/\CL(\pi) \cong  \DT(\pi)/\CT(\pi)$ of Proposition \ref{prop:1} we see that $\dot{e}_L$ is just the image of $e \in \DT(\pi)$ in the quotient, which is an idempotent since $\dot{e}$ is so.
It is then clear that $\eta^{-1}_*\iota_*(z)=[\dot{e}_L]$.

Define a $*$-homomorphism $\ell:\mathbb{C}\to \DL(\pi)/\CL(\pi)$ by $\ell(1)=\dot{e}_L$ and set $S=C_0(0,1)$.
Then $(\beta\circ \alpha )(y)$ is represented by the composition of the asymptotic homomorphisms from the following diagram.
\begin{equation}\label{eqn:am}
\xymatrix{ S \otimes \mathbb{C}\otimes A \ar[r]^-{1 \otimes \ell \otimes 1} & S \otimes \DL(\pi)/\CL(\pi) \otimes A\ar[r]^-{ \delta_t \otimes 1}  &\CL(\pi)\otimes A \ar[r]^{\varphi_s}& K(H_B),
}
\end{equation}
where here and throughout the rest of the proof the tensor products are maximal ones, and the map labelled $\delta_t$ is defined by taking the product with a canonical element $\delta$ of $E_{1,sep}(\DL(\pi)/\CL(\pi),\CL(\pi))$ associated to the extension
\[ 0 \to \CL(\pi)\to \DL(\pi) \to \DL(\pi)/\CL(\pi) \to 0\]
that we now discuss.
Fixing a separable $C^{*}$-subalgebra $\dot{M}$ of $\DL(\pi)/\CL(\pi)$, the image of $\delta$ in $E_1(\dot{M},\CL(\pi))$ is defined as follows.  Choose a separable $C^{*}$-subalgebra $M$ of $\DL(\pi)$ that surjects onto $\dot{M}$, and for each $\dot{m}\in \dot{M}$ choose a lift $m\in M$.  Let $(v_t)_{t\in T}$ be a positive, contractive, and continuous approximate unit for $M\cap \CL(\pi)$ which is quasicentral in $M$.  Then for $g\in S=C_0(0,1)$, $\delta$ is characterized by stipulating that $\delta_t(g\otimes \dot{m})$ satisfies
$$
\delta_t(g\otimes \dot{m})\sim g(v_t)m
$$
(the choices of $(v_t)$ and the various lifts do not matter up to homotopy).  In our case, to compute the composition we need, let $M$ be a separable $C^*$-subalgebra of $\DL(\pi)$ containing $e$ and $x$, and let $(v_t)$ be an approximate unit for $M\cap C_L(\pi)$ that is quasicentral in $M$.

On the level of elements, we can now concretely describe the composition in line \eqref{eqn:am} as follows.
If $g\in S=C_0(0,1)$ and $a\in A$, then under the asymptotic morphism
$\{\mu_t: SA \to K(H_B)\}_{t}$ defined by diagram \eqref{eqn:am},
elementary tensors $g\otimes a$ are mapped as follows
\begin{equation}\label{eq:nn} g\otimes a \xmapsto{\quad} g\otimes \dot{e}_L \otimes a \xmapsto{\,\,\delta_t \,\,} g(v_t)(1-x)e \otimes a\xmapsto{\,\,\varphi_{s(t)} \,\,} g(v_t(s(t)))\big(1-x (s(t) ) \big)e \pi(a)\end{equation}
for any positive map $t\mapsto s(t)$  which increases to $ \infty$ sufficiently fast.
Since the map $t\mapsto x(t)$ is an approximate unit of $K(H_B)$, $(1-x)y \in C_0(T,K(H_B))$ for all $y\in K(H_B)$.
In particular it follows that
 $\big(1-x (s(t) ) \big)e[e,\pi(a)]\sim 0$ since $[e,\pi(a)]\in K(H_B)$.
 Since $e\pi(a)=e\pi(a)e+e[e,\pi(a)]$, it follows from \eqref{eq:nn} that
\begin{equation}\label{eqn:am1}
\mu_t(g\otimes a)\sim g(v_t(s(t)))\big(1-x (s(t) ) \big)e \pi(a)e.
\end{equation}

On the other hand, the natural map $KK_1(A,B) \to E_1(A,B)$, maps $y$ to $[[\gamma_t]]$, where $\{\gamma_t: S\otimes A \to K(H_B)\}_t$ is described in \cite{Con-Hig:etheory}
as follows. Consider the extension:
\[0 \to K(H_B) \to e\pi(A) e+K(H_B) \to A \to 0.\]
 Let $(u_t)_{t\in T}$ be a contractive, positive, and continuous approximate unit of $K(H_B)$ which is quasicentral in $e\pi(A) e+K(H_B)$. Then
 \[\gamma_t(g \otimes a)\sim g(u_t)e\pi(a)e.\]
Applying Lemma~\ref{lem:qc-unit} (this time with $D$ the $C^{*}$-subalgebra of $C_u(T,L(H_B))$ generated by $e$, $\pi(A)$, $K(H_B)$, and $t\mapsto x(s(t))$), we can choose $(u_t)_{t}$ such that
 $\lim_{t\to \infty} (1-u_t) x(s(t))=0$. Since the $C^{*}$-algebra $C_0[0,1)$ is generated by the function $f(\theta)= 1-\theta$, it follows that $\lim_{t\to \infty} g(u_t) x(s(t))=0$ for all $g\in C_0[0,1)$, and in particular for all $g\in C_0(0,1)$.

 Our goal now is to verify that $(\mu_t)_t$ is homotopic to $(\gamma_t)_t$.
 Due to the choice of $(u_t)_{t}$ and the comments above, we have that
 \begin{equation}\label{eqn:am2}
 \gamma_t(g\otimes a)\sim g(u_t)e\pi(a)e \sim g(u_t)\big(1-x (s(t) ) \big)e\pi(a)e,
 \end{equation}
 for all $a\in A$ and $g\in C_0(0,1)$.   Finally, define $w_t^{(r)}=(1-r)\,v_t(s(t))+r\,u_t$, $0\leq r\leq  1$.  As
 \[
 \left[ g \big( w_t^{(r)} \big), \big(1-x (s(t) ) \big) e \pi(a) e \right] \to 0 \text{ as } t\to\infty
 \]
 for all $r\in [0,1]$ and $a\in A$, there is an asymptotic morphism  $H_t: SA \to C[0,1]\otimes K(H_B)$ defined by the condition
 \[H_t^{(r)}(g\otimes a)\sim g \big(w_t^{(r)}\big) \big(1-x (s(t) ) \big) e \pi(a) e.\]  This gives the desired homotopy joining $(\mu_t)_t$ with $(\gamma_t)_t$.
\end{proof}

As suggested by the referee, we finish this section by sketching another proof which is maybe a little less self-contained, but more conceptual.  The proof below is analogous to the approach used by Qiao and Roe to establish \cite[Proposition 4.3]{Qiao-Roe}.   The basic idea in their approach is to apply naturality of the connecting map in $E$-theory for the diagram of strictly commutative asymptotic morphisms
$$
\xymatrix@C=0.59cm{ 0 \ar[r] &  \CL(\pi)\otimes_{\max} A \ar[r] \ar[d]^-{\varphi_t} & \DL(\pi)\otimes_{\max}A \ar[r] \ar[d]^-{{\phi_t}} & (\DL(\pi)/\CL(\pi))\otimes_{\max} A \ar[r] \ar[d]^-{\bar{\phi}_t} & 0 \\
0 \ar[r] & K(H_B) \ar[r] & L(H_B) \ar[r] & L(H_B)/K(H_B) \ar[r] & 0 ~,}
$$
where $\phi_t$ and $\varphi_t$ represent the asymptotic morphisms induced by the $*$-homomorphisms $\mathbf{\Phi}$ and ${\boldsymbol\varphi}$ from the beginning of this section. The family $\bar{\phi}_t$ is the quotient family induced by $\phi_t$, and consists of $*$-homomorphisms.
Naturality of the boundary map in $E$-theory in this case amounts to the equality
\begin{equation}\label{nat-bound}[[\varphi_t]]\circ [[\delta_t \otimes \text{id}_A]]=[[\gamma_t]]\circ [[\bar{\phi}_t]], \end{equation}
where $\delta_t$ is the boundary map for the top sequence of the diagram before tensoring with $A$ and $\gamma_t$ is the boundary map for the bottom sequence. See \cite[Lemme 10]{Con-Hig:etheory} for the definition of the boundary maps associated to extensions (here and elsewhere below one should use limits to deal with the non-separable algebras involved in the way discussed earlier in this section). The naturality property of the  boundary map
  with respect to general asymptotic morphisms that was discussed in \cite[Thm. 5.3]{Guentner} seems to be the closest statement in the literature to the equality in line \eqref{nat-bound}, but it is nonetheless not sufficiently general to justify the equality. However, one can combine the arguments from
  the  second part of the proof of Theorem~\ref{inverse the} with those from \cite{Guentner} to verify naturality in full generality and in particular to justify \eqref{nat-bound}.

 Now \eqref{nat-bound} allows us to conceptualize the proof of Theorem ~\ref{inverse the}. Let $y\in KK_{i}(A,B)$ and let $z=Py\in K_{i+1}(\D(\pi)/\C(\pi))$ be its image under the Paschke duality isomorphism $P:KK_{i}(A,B)\to K_{i+1}(\D(\pi)/\C(\pi))$.
 Consider $\eta^{-1}_*\iota_*(z) \in K_{i+1}(\DL(\pi)/\CL(\pi))\cong E_{i+1}(\mathbb{C},\DL(\pi)/\CL(\pi))$, where the maps $\iota_*$ and $\eta_*$ are  isomorphisms as in the proof of Theorem \ref{thm:1}.
  We may  view  $\eta^{-1}_*\iota_*(z)\otimes[[\text{id}_{A}]]$ as an element of $E_{i+1}\left(A,\DL(\pi)/\CL(\pi)\otimes_{\max}A\right)$.
From \eqref{nat-bound} we obtain that
\begin{equation}\label{nat-bound-z}[[\varphi_t]]\circ [[\delta_t \otimes \text{id}_A]]\circ (\eta^{-1}_*\iota_*(z)\otimes[[\text{id}_{A}]])=[[\gamma_t]]\circ [[\bar{\phi}_t]]\circ (\eta^{-1}_*\iota_*(z)\otimes[[\text{id}_{A}]]). \end{equation}
 The left hand side of \eqref{nat-bound-z} represents the element  $(\beta\circ \alpha) (y)$ of $E_{i}(A,B)$ by the very definition of $\alpha$ and $\beta$.

In order to identify the right hand side of \eqref{nat-bound-z}, it is useful to note that  each individual map $\bar{\phi}_t$ is a $*$-homomorphism given by $\kappa \circ (\operatorname{ev}_t \otimes \text{id}_A )$, where
\[
	\operatorname{ev}_t \colon \DL(\pi)/\CL(\pi) \to \D(\pi)/\C(\pi)
\]
is the evaluation map at $t$ and
\[
	\kappa \colon \left(\D(\pi)/\C(\pi)\right)\otimes_{\max}A\to L(H_B)/K(H_B) , ~ [b] \otimes a \mapsto [b \cdot \pi(a)]
\]
is the ``multiplication" $*$-homomorphism. Thus the asymptotic morphism $\{\bar{\phi}_t\}$ is homotopic to the constant asymptotic morphism given by $\bar{\phi}_0$, which is equal to $\kappa \circ (\operatorname{ev}_0 \otimes \text{id}_A )$. Hence the right hand side of \eqref{nat-bound-z} is equal to
\[
	[[\gamma_t]]\circ [[\kappa]] \circ ((\operatorname{ev}_0)_* \eta^{-1}_*\iota_*(z)\otimes[[\text{id}_{A}]]) .
\]
It follows from the following commutative diagram of $*$-homomorphisms
\[
\xymatrix{
	\D(\pi) / \C(\pi) \ar[r]^{\operatorname{id}} \ar[d]^{\iota} & \D(\pi) / \C(\pi) \\
	\DT(\pi) / \CT(\pi) \ar[ur]^{\operatorname{ev}_{0}} & \DL(\pi) / \CL(\pi) \ar[l]^{\eta} \ar[u]^{\operatorname{ev}_{0}}
}
\]
that $(\operatorname{ev}_0)_* \eta^{-1}_*\iota_*(z) = z$. This allows us to simplify the right hand side of \eqref{nat-bound-z} further to
\[
	[[\gamma_t]]\circ [[\kappa]] \circ (z \otimes[[\text{id}_{A}]])
\]
where $z$ is viewed as an element in $E_{i+1}(\mathbb{C},\D(\pi)/\C(\pi))$. This can be seen to be equal to the image of $y$ under the natural map $KK_{i}(A,B)\to E_{i}(A,B)$.

Indeed, focusing on the odd case, where $y\in KK_1(A,B)$ and $z=Py\in K_0(\D(\pi)/\C(\pi))$, we may choose $e\in \D(\pi)$ as in the first part of the proof of Theorem~\ref{inverse the}, such that $z=[\dot{e}]\in K_0(\D(\pi)/\C(\pi))$. Then the $*$-homomorphism $a \in A \mapsto [ e \cdot \pi(-) ] \in L(H_B)/K(H_B) $, which represents $[[\kappa]] \circ (z \otimes[[\text{id}_{A}]])$, is the Busby invariant of the extension corresponding to $e\in \D(\pi)$. Hence its composition with the asymptotic morphism $\{\gamma_t\} \colon L(H_B)/K(H_B) \to K(H_B)$ represents the image of $y$ under the natural map $KK_1(A,B)\to E_1(A,B)$.

\bibliographystyle{abbrv}

\begin{thebibliography}{10}

\bibitem{Con-Hig:etheory}
A.~Connes and N.~Higson.
\newblock D\'eformations, morphismes asymptotiques et ${K}$-th\'eorie
  bivariante.
\newblock {\em C. R. Acad. Sci. Paris S\'er. I Math.}, 311(2):101--106, 1990.

\bibitem{DadEil:AKK}
M.~Dadarlat and S.~Eilers.
\newblock Asymptotic unitary equivalence in {$KK$}-theory.
\newblock {\em $K$-Theory}, 23(4):305--322, 2001.

\bibitem{DadEil:class}
M.~Dadarlat and S.~Eilers.
\newblock On the classification of nuclear ${C}^*$-algebras.
\newblock {\em Proc. London Math. Soc. (3)}, 85(1):168--210, 2002.

\bibitem{DadLor:unsusp}
M.~Dadarlat and T.~A. Loring.
\newblock ${K}$-homology, asymptotic representations, and unsuspended
  ${E}$-theory.
\newblock {\em J. Funct. Anal.}, 126(2):367--383, 1994.

\bibitem{HigRoe:khomology}
N.~Higson and J.~Roe.
\newblock {\em {Analytic ${K}$-homology}}.
\newblock Oxford University Press, Oxford, 2000.
\newblock Oxford Science Publications.

\bibitem{Guentner}
E.~Guentner.
\newblock Relative $E$-theory.
\newblock {\em $K$-Theory}, 17(1):55--93, 1999.

\bibitem{Kas:cp}
G.~G. Kasparov.
\newblock {Hilbert $C^*$-modules: Theorems of Stinespring and Voiculescu}.
\newblock {\em J. Operator Theory}, 4:133--150, 1980.

\bibitem{Kas:KK}
G.~G. Kasparov.
\newblock {The operator $K$-functor and extensions of {$C^*$}-algebras}.
\newblock {\em Math. USSR--Izv.}, 16:513--672, 1981.
\newblock English translation.

\bibitem{Paschke}
W.~L. Paschke.
\newblock {$K$}-theory for commutants in the {C}alkin algebra.
\newblock {\em Pacific J. Math.}, 95(2):427--434, 1981.

\bibitem{Qiao-Roe}
Y.~Qiao and J.~Roe.
\newblock On the localization algebra of {G}uoliang {Y}u.
\newblock {\em Forum Math.}, 22(4):657--665, 2010.

\bibitem{Skandalis:K-nuclear}
G.~Skandalis.
\newblock {Une Notion de Nucl\'earit\'e en K-th\'eorie}.
\newblock {\em $K$-theory}, 1(5):549--573, 1988.

\bibitem{Thomsen}
K.~Thomsen.
\newblock On absorbing extensions.
\newblock {\em Proc. Amer. Math. Soc.}, 129(5):1409--1417, 2001.

\bibitem{Voi:Weyl-vn}
D.~Voiculescu.
\newblock {A non-commutative Weyl-von Neumann theorem}.
\newblock {\em Rev. Roum. Math. Pures et Appl.}, 21:97--113, 1976.

\bibitem{Yu:localization-algebra}
G.~Yu.
\newblock Localization algebras and the coarse {B}aum-{C}onnes conjecture.
\newblock {\em $K$-Theory}, 11(4):307--318, 1997.

\end{thebibliography}

\end{document}